\documentclass[psamsfonts]{amsart}

\usepackage{amssymb,amsfonts}
\usepackage[all,arc]{xy}
\usepackage{enumerate}
\usepackage{mathrsfs}
\usepackage{tikz}
\usepackage{graphicx} % Allows including images
\usepackage{amsmath}
\usepackage{bbm}
\usepackage{pgf,tikz,pgfplots}
\usepackage[english]{babel}
\usepackage[utf8]{inputenc}
\usepackage[super]{nth}
\usepackage{amsmath}
\usepackage{mathrsfs}
\usetikzlibrary{arrows}
\newtheorem{thm}{Theorem}[section]
\newtheorem{cor}[thm]{Corollary}
\newtheorem{prop}[thm]{Proposition}
\newtheorem{lem}[thm]{Lemma}

\theoremstyle{definition}
\newtheorem{defn}[thm]{Definition}

\newtheorem{obs}[thm]{Observation}
\newtheorem{clm}[thm]{Claim}

\newtheorem{col}[thm]{Corollary}
\newtheorem{rem}[thm]{Remark}

\def\p{{\mathfrak p}}
\def\a{{\mathfrak a}}

\def\n{{\mathfrak n}}

\makeatletter
\let\c@equation\c@thm
\makeatother
\numberwithin{equation}{section}

\title{Symbolic powers of  multipartite hypergraphs and Waldschmidt constant  }
\author[A. Alilooee]{Ali Alilooee}
\address{Bradley University, Department of Mathematics and Statistics,
Bradley Hall,
Peoria, IL, USA }
\email{a20480m2018@gmail.com}
\author[A. Banerjee]{Arindam Banerjee}
\address{Ramakrishna Mission Vivekenanda University, Belur, West Bengal, India}
\email{123.arindam@gmail.com}

\begin{document}

\maketitle

%%%%%%%%%%%%%%%%%%%%%%%%%%%%%%%%%%%%%%%%%%%%%%%%%%%%%%%%%%%%%%%%%%%%%%%%%%%%%%
\begin{abstract}
In this paper, we find a criterion to answer when the second symbolic and ordinary powers
of the edge ideal of a 3-partite hypergraph are equal. Also, we give the formula to compute
the Waldschmidt constant of the path ideals of a cycle.    
\end{abstract}
%%%%%%%%%%%%%%%%%%%%%%%%%%%%%%%%%%%%%%%%%%%%%%%%%%%%%%%%%%%%%%%%%%%%%%%%%%%%%%

%%%%%%%%%%%%%%%%%%%%%%%%%%%%%%%%%%%%%%%%%%%%%%%%%%%%%%%%%%%%%%%%%%%%%%%%%%%%%%
\section{introduction}
%%%%%%%%%%%%%%%%%%%%%%%%%%%%%%%%%%%%%%%%%%%%%%%%%%%%%%%%%%%%%%%%%%%%%%%%%%%%%%
\hspace{.1 in} In this paper, we study the symbolic power of some classes of squarefree monomial ideals. 
Let $I$ be an ideal in a Noetherian domain $R$, its $n^{\text{th}}$ symbolic power is defined as follow
$$I^{(n)}=\displaystyle\bigcap_{\p\in \text{Ass}_{R}(R/I)}(I^{n}R_{\p}\cap R).$$

Two of the most challenging questions about the symbolic powers are:
\begin{enumerate}
\item In general, the symbolic and ordinary powers are not equal. The question is for what classes of ideals, we have the equality of the symbolic and the ordinary powers. It is a too broad question. Another less overall problem here is for what powers $n$ we have 
$I^{n}=I^{(n)}$?
\item Another interesting questions is to find the generators of the symbolic powers of some classes of ideals. More precisely, if $I$ is a squarefree monomial ideal, we know there is a combinatorial presentation of the generators of $I$. The question is what the combinatorial character can describe the generators of $I^{(n)}$ for some $n$?
\end{enumerate}
In this paper, we address these questions for a particular class of cubic squarefree monomial ideal. More accurately, we consider the class of 3-uniform, 3-partite hypergraph, and we will give a combinatorial description of the generators of the second symbolic power of these hypergraphs.
We define the notion of bad hypergraphs and classify all 3-uniform 3-partite hypergraphs for which second symbolic power and second ordinary powers are the same:
 
\begin{thm}
Let $\mathcal{H}$ be a 3-uniform and 3-partite hypergraph and $I=I(\mathcal{H})$ be its edge ideal. Then $I^{(2)}=I^2$ if and only if one of the following holds
\begin{enumerate}
\item $\mathcal{H}$ has no bad subhypergraph of length 3;
\item If $\mathcal{H}$ has a bad subhypergraph of length 3 say $\mathcal{B}$ then there is a hyperedge $E$ of $\mathcal{B}$ such that $V(\mathcal{B})\backslash E$ is a hyperedge of $\mathcal{H}$
\end{enumerate}
\end{thm}

%{Let $I$ be an ideal in a Noetherian domain $R$, its $n^{\text{th}}$ symbolic power is defined as follow
%$$I^{(n)}=\displaystyle\bigcap_{\p\in \text{Ass}_{R}(R/I)}(I^{n}R_{\p}\cap R).$$
% We have the fundamental Zarisiki-Nagata theorem \cite{Nagata1962}, \cite{Zariski1948}.
%\begin{thm}[\cite{SULLIVANT2006}, Theorem 1.1]
 %If $I$ is a radical ideal in a $k[x_1,\dots,x_k]$ over an algebraically closed field $k$ then we have 

%$$I^{(n)}=\displaystyle\left(f: \text{$\frac{\partial^{|a|}f}{\partial x^{a}}\in I$, for all $a\in \mathbb{N}^{k}$ with $|a|=
%\sum_{i=1}^{k}a_i\leq n-1$}\right).$$
%\end{thm} 
%Then $I^{(n)}$ contain all polynomials that vanish of order $n$ in the variety $V(I)$. We can have another way to explain the $n$-th symbolic power of the ideal $I$.
%\begin{prop}[\cite{Villarreal2015}, Definition 4.3.22]
%Let $I$ be an ideal in the ring $R$ and $\p_1,\p_2,\dots,\p_r$ be the minimal primes of $I$. Then for all $n\geq 1$ we have 
%$$I^{(n)}=\q_1\cap\q_2\cap\dots\cap\q_r$$
%where $\q_i$ is the primary component of $I^{n}$ corresponding to $\p_i$.  
%\end{prop}}

 The ``ideal containment problem'',  given a nontrivial homogeneous ideal I of a polynomial ring  over a field, the problem is to determine all positive integer pairs $(m, r)$ such that $I^{(m)} \subseteq I^r$.  This
problem was motivated by the fundamental results of Ein-Lazarsfeld-Smith and Hochster-Huneke showing that containment
holds whenever $m \geq r(n-1)$. In order to capture more precise information about these containments, Bocci and Harbourne introduced the resurgence of $I$, denoted $\rho(I)$ and defined as $\rho(I) = \text{sup}\{\frac{m}{r}| I
^{(m)} \subsetneq I^r\}$.
In general, computing $\rho(I)$ is quite difficult. There has been an ongoing
research programme to bound $\rho(I)$ in terms of other invariants of $I$ that may be easier
to compute. One of the important constant for an ideal is the {\bf Waldschmidt constant}. Let $I$ be a nonzero homogenous ideal of the graded ring $R$, we let $\alpha(I)=\min\{d:I_{d}\neq 0\}$. The  Waldschmidt constant of $I$ is defined as follow
$$\widehat{\alpha}(I)=\displaystyle\lim_{n\rightarrow\infty}\frac{\alpha(I^{(n)})}{n}.$$
Waldschmidt in \cite{Waldschmidt1977} showed that this limit exists for ideals of finite point sets
in the context of complex analysis. He was interested in finding the minimum degree of a hypersurface that passes through the set of points with multiplicities. In other words, he was studying $\alpha(I^{(n)})$ when $I$ defined a set of points.

We also study this in this paper, and our main result regarding this is as follows:

\begin{prop}
Let $\mathcal{H}$ be a simple $r$-uniform and $r$-partite hypergraph and $I=I(\mathcal{H})$ be its edge ideal. Then the Waldschmidt constant of $I$ (i.e. $\widehat{\alpha}(I)$) is $r$.
\end{prop}

We organize the paper as follows.  In section (2) we review some necessary preliminaries. Then in section (3), we will consider the $r$-partite hypergraph and we show the Waldschmidt constant of their edge ideals is $r$. In section (4) we shall study cubic path ideals.

%%%%%%%%%%%%%%%%%%%%%%%%%%%%%%%%%%%%%%%%%%%%%%%%%%%%%%%%%%%%%%%%%%%%%%%%%%%%%%
%%%%%%%%%%%%%%%%%%%%%%%%%%%%%%%%%%%%%%%%%%%%%%%%%%%%%%%%%%%%%%%%%%%%%%%%%%%%%%
\section{Preliminaries} 
%%%%%%%%%%%%%%%%%%%%%%%%%%%%%%%%%%%%%%%%%%%%%%%%%%%%%%%%%%%%%%%%%%%%%%%%%%%%%%
%%%%%%%%%%%%%%%%%%%%%%%%%%%%%%%%%%%%%%%%%%%%%%%%%%%%%%%%%%%%%%%%%%%%%%%%%%%%%%

\begin{defn}
Let $X=\{x_1,x_2,\dots,x_n\}$ be a finite set. A {\bf simple hypergraph} on $X$ is a family $\mathcal{H}=(E_1,\dots,E_q)$ of subsets of $X$ such that 
\begin{enumerate}
\item All $E_i$s are nonempty;\\
\item $\displaystyle\cup_{i=1}^{q} E_i=X$;\\
\item None of the $E_i$s is contained within another.

\end{enumerate}
In this paper all hypergraphs are simple. The elements of $X$ are called {\bf vertices} and the sets $E_1,E_2,\dots,E_q$ are called the {\bf hyperedges} of $\mathcal{H}$. 
We sometimes denote the hypergraph $\mathcal{H}$ over the vertex set $X$ and with the edge set $\mathcal{E}=\{E_1,E_2,\dots,E_q\}$ with the pair $\mathcal{H}=(X,\mathcal{E})$.   
\end{defn}
\begin{defn}
A hypergraph $\mathcal{H}$ is called {\bf $r$-uniform} if for each $E\in \mathcal{H}$ we have $|E|=r$. It is obvious that a 2-uniform hypergraph is a graph.  
\end{defn}
\begin{defn}
Let $\mathcal{H}=(E_1,E_2,\dots,E_q)$ be a simple $r$-uniform hypergraph over the vertex set $X$ and $k$ be a field. We define the {\bf edge ideal} of $\mathcal{H}$ as a squarefree
monomial  ideal $I(\mathcal{H})$ in the polynomial ring $R=k[X]$ defined as
$$I(\mathcal{H})=(x_{i_1}\dots x_{i_r}:\{i_1\dots i_r\}\in \mathcal{H}).$$
\end{defn}

We need the following result in Chapter 4. We do not give the exact definition of the fractional coloring for more details for the concept of the fractional coloring see \cite{Ullman2011}. 
\begin{defn}[Transitivity]\label{transitivity}
A hypergraph is called {\bf vertex-transitive} if for each two vertices $x,y$ there exists an automorphism  
$\pi$ on the hypergraph such that $\pi(x)=y$. 
\end{defn}
\begin{prop}[(\cite{Ullman2011}, Proposition 3.1.1]Let $\mathcal{H}$ be a vertex-transitive hypergraph and $\chi^{*}(\mathcal{H})$ denotes the fractional chromatic number of $\mathcal{H}$,  then we have  

$$ \chi^{*}(\mathcal{H})= \frac{|V(\mathcal{H})|}{\a(\mathcal{H})}.$$
where $|V(\mathcal{H})|$ is the cardinality of the vertex set $\mathcal{H}$ and $\a(\mathcal{H})$ is the maximum cardinality of an independent set (A set $W$ is called independent if there is no hyperedge $E$ such that $E\subset W$).
\end{prop}
In order to prove the one of our observations in Chapter 3 we need to recall the following lemma shown by Sullivant in \cite{SULLIVANT2006}.

\begin{lem}[\cite{SULLIVANT2006}, Lemma 3.6]\label{sulivan}
Let $\mathcal{H}$ be a hypergraph and $I=I(\mathcal{H})$ be its edge ideal over the polynomial ring $R=k[x_1,\dots,x_n]$. Then 
\begin{itemize}
\item A monomial $X^{a}=x_1^{a_1}\dots x_n^{a_n}\in I^{(m)}$ for some $m\geq 1$ and the integer vector $a=(a_1,\dots,a_n)$, if and only if  for every monomial $X^{b}$ (where $b$ is an integer vector) with
$\text{deg}(X^b)\leq m-1$ dividing $X^{a}$, $\frac{X^{a}}{X^b}\in I$. 
\end{itemize}

\end{lem}

%We recall the definition of a balanced hypergraph from \cite{Faridi2011}. 
%\begin{defn}
%Let $\mathcal{H}$ be a hypergraph on $X$ and let $k\geq 2$ be an integer. A cycle of length $k$ is a sequence 
%$(x_1,E_1,x_2,E_2,x_3,\dots,x_k,E_k,x_1)$ with; 
%\begin{itemize}
%\item $E_1,E_2,\dots,E_k$ are distinct edges of $\mathcal{H}$;
%\item $x_1,x_2,\dots,x_k$ are distinct vertices of $\mathcal{H}$;
%\item $x_i,x_{i+1}\in E_i$ for $(i=1,2,\dots,k-1)$
%\item $x_k,x_1\in E_k$.
%\end{itemize}
%A cycle is called odd (or even) if $k$ is odd (or even). A hypergraph $\mathcal{H}$ is called {\bf balanced} if every odd cycle has a hyperedge containing three vertices of the cycle. 
%\end{defn} 
%We also need the definition of a simplicial tree from  Caboara and Faridi in \cite{Faridi2011}.
%\begin{defn}
%Let $\mathcal{H}$ be a hypergraph and $F$ be a hyperedge of $\mathcal{H}$. Then $F$ is called a {\bf leaf} of $\mathcal{H}$ 
%if either $F$ is the only hyperedge of $\mathcal{H}$ or else there exists some hyperedge $G\neq F$ such that all hyperedges 
%$H\neq F$ we have $H\cap F\subset G$. A hypergraph is called a {\bf simplicial cycle} if it has no leaf but every its nonempty proper subhypergraph has a leaf. A connected hypergraph is called a {\bf simplicial tree} if it has no simplicial cycle.  

%\end{defn}
We recall the definition of Mengerian hypergraphs here. 
\begin{defn}\label{min-max}
Let $A$ be the edge-vertex incidence matrix of the hypergraph $\mathcal{H}$. Then $\mathcal{H}$ is called a {\bf Mengerian hypergraph} if for all $c\in \mathbb{N}^{n}$ we have   
 
$$\min\{c\cdot x|Ax\geq \mathbbm{1},x\in\mathbb{N}^{n}\}=\max\{y\cdot\mathbbm{1}|yA\leq c; y\in \mathbb{N}^{m}\}$$

\end{defn}
%We will use the following results in this paper.
%\begin{thm}[\cite{Herzog2008}, Theorem 3.2]
%Every simplicial tree is Mengerian.
%\end{thm}
%\begin{thm}[\cite{Herzog2008}, Theorem 2.5]\label{Herzog}Every balanced hypergraph is Mengerian. 
%\end{thm}
\begin{thm}[\cite{Villarreal2007}, \cite{Herzog2008}, \cite{Tai2019}, \cite{Trung2006}]\label{Tai}
Let $I$ be the edge ideal of a hypergraph $\mathcal{H}$. Then $I^{(n)}=I^{n}$ for all $n\geq 1$ if and only if $\mathcal{H}$ is Mengerian.
\end{thm}
\begin{prop}[\cite{Berge1989}, page 199, Proposition 1 and 2]
Let $I$ be a squarefree monomial ideal and $x$ be a variable and assume $I^{(n)}=I^{n}$ for all $n\geq 1$, then $(I:x)^{(n)}=(I:x)^{n}$.
 \end{prop}
We also recall the following remark from \cite{Dao2018}. We will use this remark to prove one of our main result in Chapter 3. 
\begin{rem}[\cite{Dao2018}, Remark 4.12]\label{Craigmethod}
Let $J$ be a squarefree monomial ideal and let $x$ be a variable. We let $I_{x}=(J:x)$ and $I$ be an ideal generated by all monomials in $J$ which do not involve $x$. Put $I_x=I+L$ where $L$ is an ideal generated by all monomials in $I_x$ which are not in  $I$. Then if we assume that $I^{(n)}=I^{n}$ and 
$I_{x}^{(n)}=I_{x}^{n}$ for all $n\geq 1$, then we have $J^{(n)}=J^n$ for all $n\geq 1$ if and only if 
$$\displaystyle I^k\cap I^{i}L^{n-i}\subset \sum_{j=k}^{n} I^{j}L^{n-j}=I^{k}I_{x}^{n-k}$$
for all $k$ and $i$ in which we have $0\leq i<k\leq n$.  
\end{rem}

%%%%%%%%%%%%%%%%%%%%%%%%%%%%%%%%%%%%%%%%%%%%%%%%%%%%%%%%%%%%%%%%%%%%%%%%%%%%%%%%%%%%%%%%%
\section{$r$-Partite Hypergraphs}
%%%%%%%%%%%%%%%%%%%%%%%%%%%%%%%%%%%%%%%%%%%%%%%%%%%%%%%%%%%%%%%%%%%%%%%%%%%%%%%%%%%%%%%%%%

\begin{defn}
Let $\mathcal{H}$ be a uniform hypergraph. The hypergraph $\mathcal{H}$ is called {\bf $r$-partite} if $V(\mathcal{H})=\displaystyle\bigcup_{k=1}^{r}X_i$ such that
\begin{enumerate}
\item $(X_1,\dots,X_r)$ are pairwise disjoint.
\item For each $i$, $X_i$ is a vertex cover of $\mathcal{H}$. (A set of vertices is called a {\bf vertex cover} of $\mathcal{H}$ if it meets all the edges.) 
\item For each hyperedge $E$ of $\mathcal{H}$ and for each $i$ we have $|X_i\cap E|=1$.
\end{enumerate}
Moreover a $r$-partite hypergraph $\mathcal{H}$ with the $r$-partition $(X_1,X_2,\dots,X_{r})$ is called a {\bf complete $r$-partite hypergraph} if for each $x_i\in X_i$, for $i=1,2,\dots,r$, we have $\{x_1,x_2,\dots,x_r\}$ is a hyperedge of $\mathcal{H}$. 
\end{defn}
Beckenbach and Scheidweiler in \cite{Isabel2017} showed that every Mengerian $r$-uniform hypergraph is $r$-partite. Then we can have the following theorem.      
\begin{thm}[\cite{Isabel2017},Theorem 2.1]\label{Isabel}
Let $\mathcal{H}$ be a $r$-uniform hypergraph and $I=I(\mathcal{H})$ be its edge ideal. If $I^{(n)}=I^n$ for all $n\geq 1$, then $\mathcal{H}$ is $r$-partite.  

\end{thm}
Now we are ready to show complete $r$-uniform, $r$-partite hypergraphs are Mengerian. First we show the following lemma.
\begin{lem}\label{newLemma}
Let $\mathcal{H}$ be a $r$-uniform and complete $r$-partite hypergraph such that at least one of the partitions of $\mathcal{H}$ has exactly one element. If we let $J=I(\mathcal{H})$ be the edge ideal of $\mathcal{H}$, then $J^{(n)}=J^n$ for all $n\geq 1$. 

\end{lem} 
\begin{proof}
Let $V(\mathcal{H})=\displaystyle\bigcup_{i=1}^{r} X_i$ be a $r$-partition and let $X_1=\{x_1\}$. We proceed the proof by using the induction on $r$. If $r=2$, then $\mathcal{H}$ is a complete bipartite graph and then from [\cite{Villareal2009}, Proposition 4.27] we have $J^{(n)}=J^{n}$ for all $n\geq 1$. So we assume $r>2$. Let $I_{x_1}=(J:x_1)$. Since $\mathcal{H}$ is a complete $r$-partite hypergraph we can say $I_{x_1}$ is the edge ideal of a complete $(r-1)$-partite hypergraph over $(X_2,\dots,X_{r})$ and therefore, by using the induction hypothesis we have 
\begin{align*}
I_{x}^{(n)}=I_{x}^{n}\hspace{0.2 in}\text{for all $n\geq 1$}.
\end{align*} 
We use the notations of Remark~\ref{Craigmethod}. Since $\mathcal{H}$ is complete $r$-partite hypergraph and $X_1=\{x_1\}$, if $I$ is an ideal generated by monomials in $J$ which not involving $x_1$  we have $I=0$. Then since we have $I_{x_1}^{(n)}=I_{x_1}^{n}$ 
and $I^{(n)}=I^{n}$ for all $n\geq 1$ if we let $L=(J:x_1)$, we have    
$$\displaystyle I^k\cap I^{i}L^{n-i}\subset \sum_{j=k}^{n} I^{j}L^{n-j}$$
for all $k$ and $i$ in which we have $0\leq i<k\leq n$.  
Then from Remark~\ref{Craigmethod} we have $J^{(n)}=J^{n}$ for all $n\geq 1$. 

\end{proof}

\begin{thm}\label{complete}
Let $\mathcal{H}$ be a $r$-uniform and complete $r$-partite hypergraph and $J=I(\mathcal{H})$ be its edge ideal. Then $J^{(n)}=J^n$ for all $n\geq 1$. 

\end{thm}
\begin{proof}
Let $R=k[V(\mathcal{H})]$ be a polynomial ring over a field $k$ and 
$$V(\mathcal{H})=\displaystyle\bigcup_{m=1}^{r} X_m$$ 
be a $r$-partition. We proceed the proof by using the induction on $r$ and induction on $|X_1|$. If $r=2$, then $\mathcal{H}$ is a complete bipartite graph and then from [\cite{Villareal2009}, Proposition 4.27] we have $J^{(n)}=J^{n}$ for all $n\geq 1$. 
Form Lemma~\ref{newLemma}, if we assume $|X_1|=1$, then we have $J^{(n)}=J^{n}$ for all $n\geq 1$.

So we assume $r>2$ and $|X_m|>1$ for all $m$. We pick a variable $x$ in $V(\mathcal{H})$ and we consider $I_{x}=(J:x)$. Without loss of generality we can assume $x\in X_1$. 
Since $\mathcal{H}$ is a complete $r$-partite hypergraph we have $I_{x}$ is generated by all monomials of the form $m/x$ where $m$ is a generator of $J$ contains $x$, thus we can conclude that $I_x$ is the edge ideal of a complete $(r-1)$-partite hypergraph over $(X_2,\dots,X_{r})$ and therefore, by using the induction hypothesis on $r$ we have 
\begin{align*}
I_{x}^{(n)}=I_{x}^{n}\hspace{0.2 in}\text{for all $n\geq 1$}.
\end{align*} 
On the other hands, if $I$ is an ideal generated by monomials in $J$ which not involving $x$,  we have $I$ is the edge ideal of a hypergraph $\mathcal{H}'$ where $\mathcal{H}'$ is a $r$-complete $r$-partite hypergraph on the vertex set  $V(\mathcal{H})\backslash\{x\}$ with $r$-partition $\displaystyle (X_1\backslash\{x\},X_2,\dots, X_r)$. We write $I_x=I+L$ where $L$ is an ideal generated with monomials in $I_x$ which are not in $I$. From induction hypothesis on $|X_1|$  we have $I^{(n)}=I^{n}$ for all $n\geq 1$. Also note that since $\mathcal{H}$ is $r$-complete we have $I\subset L$ and then $L=(J:x)$. We use the notations of Remark 2.14. thus we should show    
$$\displaystyle I^k\cap I^{i}L^{n-i}\subset \sum_{j=k}^{n} I^{j}L^{n-j}=I^{k}I_{x}^{n-k}=I^{k}L^{n-k}$$
for all $k$ and $i$ in which we have $0\leq i<k\leq n$. We pick a monomial $f\in I^{k}\cap I^{i}L^{n-i}$ then we have the following expressions for $f$
\begin{equation}
f=\displaystyle\left(\prod_{1\leq j\leq k}m_{j}\right)\prod_{y\in D}y
\end{equation}

\begin{equation}
f=\displaystyle\left(\prod_{1\leq j\leq i}m'_{j}\right)\left(\prod_{1\leq j\leq a}g_{j}\right)\prod_{\omega\in D'}\omega
\end{equation}
where $D,D'$ are sets of vertices, $m_j,m'_{j}\in I$, $g_i\in L$, and $a\geq n-i$.

Note that since $m_j,m'_j\in I$ and $\mathcal{H}$ is a $r$-partite hypergraph then we can conclude that the degree of the variables which are in $X_1$ in the first expression is at least $k$. Since there is no $x\in X_1$ such that $x$ divides $\prod_{1\leq j\leq a}g_j$ then we can conclude that the degree of the variables in $X_1$ in the monomial $\prod_{\omega\in D'}\omega$ is $k-i$. Then $f$ can be written as follow
$$f=\displaystyle\left(\prod_{1\leq j\leq i}m'_{j}\right)\left(\prod_{1\leq j\leq a}g_{j}\right)\left(\prod_{x\in X_1}x\right)\prod_{\omega\in D''}\omega$$
where $\text{deg}(\prod_{x\in X_1}x)=k-i$, then since $\mathcal{H}$ is a $r$-partite complete hypergraph then we have  
$$f=\displaystyle\left(\prod_{1\leq j\leq i}m'_{j}\right)\left(\prod_{1\leq j\leq k-i}x_jg_{j}\right)\left(\prod_{1\leq j\leq n-k}g_{j}\right)\prod_{\omega\in D'''}\omega$$
where $D''$ and $D'''$ are sets of vertices. Then we have $f\in I^{k}L^{n-k}$.     
Therefore, from Remark~\ref{Craigmethod} we have $J^{(n)}=J^{n}$ for all $n\geq 1$. 
\end{proof}
\begin{rem}
Theorem~\ref{complete} is not true without the assumption of completeness. For example consider the hypergraph $\mathcal{H}$ over the vertex set $V(K)=\{x_1,\dots,x_6\}$ with the hyperedges $\{x_1,x_2,x_3\},\{x_3,x_4,x_5\},\{x_5,x_6,x_2\}$. Clearly $\mathcal{H}$ is a 3-uniform, 3-partite hypegraph but not complete. And we have $I(\mathcal{H})^{(2)}=I(\mathcal{H})^{2}+(x_1x_2x_3x_4x_5x_6)$.   
\end{rem}
\begin{prop}\label{myprop}
Let $\mathcal{H}$ be a simple $r$-uniform and $r$-partite hypergraph and $I=I(\mathcal{H})$ be its edge ideal. Then the Waldschmidt constant of $I$ (i.e. $\widehat{\alpha}(I)$) is $r$.
\end{prop}
\begin{proof}
Since $\mathcal{H}$ is a $r$-partite hypergraph then we can say $V(\mathcal{H})=X_1\cup\dots \cup X_r$ where $(X_1,\dots,X_r)$ is the 
$r$-partition of $\mathcal{H}$. Pick a monomial $\n\in I^{(m)}$ where $m$ is an integer. Since each $X_i$ for $i=1,2,\dots, r$ is a vertex cover of $\mathcal{H}$, then $\n\in I(X_i)^m$ (where $I(X_i)$ is an ideal generated by all variables belonging to $X_i$) and so there is a monomial generator $\n_i\in I(X_i)^m$ such that $\n_i$ divides $\n$. Since the $X_i$s are pairwise disjoint, then the $\n_i$s are pairwise coprime and then we can say $\displaystyle\prod_{i=1}^{r} \n_i$ divides $\n$ and then we have 
$\text{degree}(\n)\geq rm$. Now since $I^{m}\subset I^{(m)}$ for each $m\geq 1$ we can say $\alpha(I^{(m)})=rm$ for each $m\geq 1$. Therefore
$$\widehat{\alpha}(I)=\displaystyle\lim_{m\rightarrow\infty}\frac{\alpha(I^{(m)})}{m}=r.$$ 
\end{proof}
To show the next corollary we need the following theorem \cite{Adm2016}.
\begin{thm}[\cite{Adm2016}, Theorem 4.6]\label{Adam'Theorem}
Let $\mathcal{H}$ be a hypergraph with a non-trivial hyperedge, and let $I = I(\mathcal{H})$ be its edge ideal. Then
\begin{align*}
\widehat{\alpha}(I)=\displaystyle\frac{\chi^{*}(\mathcal{H})}{\chi^{*}(\mathcal{H})-1}
\end{align*} 
where $\chi^{*}(\mathcal{H})$ is the fractional chromatic number of $\mathcal{H}$. 
\end{thm}
\begin{col}
Let $\mathcal{H}$ be a simple $r$-uniform, $r$-partite hypergraph with a non-trivial hyperedge. Then the fractional chromatic number of the hypergraph $\mathcal{H}$ is
$$\chi^{*}(\mathcal{H})=\displaystyle\frac{r}{r-1}.$$ 
\end{col}
\begin{proof}
Let $I$ be the edge ideal associated with $\mathcal{H}$. Theorem~\ref{Adam'Theorem} and Proposition~\ref{myprop} settle
the claim. 

\end{proof}
In order to prove one of the main results of this section we need the following definition. 
\begin{defn}
A simplicial cycle  $K$ is called the {\bf bad hypergraph of length 3} if $V(K)=\{x_1,\dots,x_6\}$ and the hyperedges of $K$ are 
$$\{x_1,x_2,x_3\},\{x_3,x_4,x_5\},\{x_5,x_6,x_2\}.$$ 
\end{defn}
\definecolor{zzttqq}{rgb}{0.6,0.2,0.}
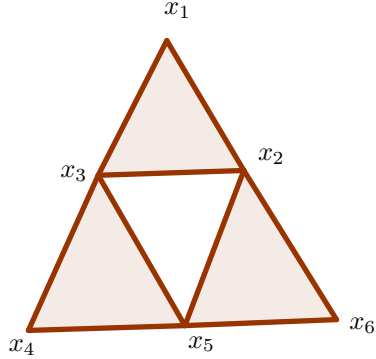
\begin{tikzpicture}[line cap=round,line join=round,>=triangle 45,x=0.8cm,y=0.8cm]
\clip(-2.3,-1.98) rectangle (14,4);
\fill[line width=2.pt,color=zzttqq,fill=zzttqq,fill opacity=0.10000000149011612] (8.7,1.16) -- (6.28,1.08) -- (7.42,3.32) -- cycle;
\fill[line width=2.pt,color=zzttqq,fill=zzttqq,fill opacity=0.10000000149011612] (6.28,1.08) -- (7.72,-1.42) -- (5.12,-1.5) -- cycle;
\fill[line width=2.pt,color=zzttqq,fill=zzttqq,fill opacity=0.10000000149011612] (8.7,1.16) -- (7.72,-1.42) -- (10.24,-1.32) -- cycle;
\draw [line width=2.pt,color=zzttqq] (8.7,1.16)-- (6.28,1.08);
\draw [line width=2.pt,color=zzttqq] (6.28,1.08)-- (7.42,3.32);
\draw [line width=2.pt,color=zzttqq] (7.42,3.32)-- (8.7,1.16);
\draw [line width=2.pt,color=zzttqq] (6.28,1.08)-- (7.72,-1.42);
\draw [line width=2.pt,color=zzttqq] (7.72,-1.42)-- (5.12,-1.5);
\draw [line width=2.pt,color=zzttqq] (5.12,-1.5)-- (6.28,1.08);
\draw [line width=2.pt,color=zzttqq] (8.7,1.16)-- (7.72,-1.42);
\draw [line width=2.pt,color=zzttqq] (7.72,-1.42)-- (10.24,-1.32);
\draw [line width=2.pt,color=zzttqq] (10.24,-1.32)-- (8.7,1.16);
\draw (7.22,4.1) node[anchor=north west] {$x_1$};
\draw (8.78,1.68) node[anchor=north west] {$x_2$};
\draw (5.5,1.4) node[anchor=north west] {$x_3$};
\draw (4.64,-1.5) node[anchor=north west] {$x_4$};
\draw (7.62,-1.44) node[anchor=north west] {$x_5$};
\draw (10.3,-1.14) node[anchor=north west] {$x_6$};
\end{tikzpicture}
\begin{obs}\label{myobs}
Let $\mathcal{H}$ be a simple 3-uniform, 3-partite hypergraph and $I=I(\mathcal{H})$ be its edge ideal. If $I^{(2)}\neq I^{2}$, then there is a subhypergraph in $\mathcal{H}$ which is isomorphic to the bad hypergraph of length 3. 
\end{obs}
\begin{proof}
To prove this observation we use the following notation. If $e=\{x_1,\dots,x_r\}$ is a hyperedge of $\mathcal{H}$ we define the monomial $e$ as follow
$$e=\prod_{i=1}^{r} x_i.$$
In this proof we freely identify the hyperedge $e$ and the monomial $e$. We also identify vertices and variables. 

Suppose $(X,Y,Z)$ is a 3-partition of $\mathcal{H}$. Since $I^{(2)}\neq I^2$ then we can pick the monomial $X^{a}\in I^{(2)}\backslash I^2$. Pick a variable $x$ dividing $X^{a}$. 
From Lemma~\ref{sulivan} we have $\frac{X^{a}}{x}\in I$. Thus there is a hyperedge $e$ in $\mathcal{H}$ such that 
$e$ divides $\frac{X^{a}}{x}$. Since $X^{a}\notin I^2$ there is a variable $y$ in $e$ such that $y^2$ does not divide $X^{a}$. 

Then we can pick a variable $y$ such that $y^2$ does not divide $X^{a}$. Therefore, by using Lemma~\ref{sulivan} we can write 
there is a hyperedge $e'$ in $\mathcal{H}$ such that 
$e'$ divides $\frac{X^{a}}{y}$. Note that $e\neq e'$ because $y\in e\backslash e'$. Also note that $e\cap e'\neq \emptyset$ because otherwise $e$ and $e'$ are coprime and then $ee'$ divides $X^{a}$ and then $X^{a}\in I^2$ which is a contradiction.  Thus we can pick a variable $z\in e\cap e'$ such that $z^2$ doesn't divide $X^a$ because otherwise $ee'$ will divide $X^{a}$ and then $X^{a}\in I^2$ again it is a contradiction.  So we pick a variable $z\in e\cap e'$ such that $z^2$ doesn't divide $X^{a}$. 
From Lemma~\ref{sulivan} we can conclude that there is a hyperedge $e''$ in $\mathcal{H}$ such that 
$e''$ divides $\frac{X^{a}}{z}$. Again note that $e''\notin \{e,e'\}$ because $z\in e\cap e'$ but $z$ doesn't belong to $e''$. Also since $X^{a}\notin I^2$ we can write $e,e'$ and $e''$ intersect pairwise. 

Now we are ready to write the following claim.
\begin{clm} 
Let $\mathcal{H}$ be a simple 3-uniform, 3-partite hypergraph and $I=I(\mathcal{H})$ be its edge ideal. If $I^{(2)}\neq I^{2}$, then the hypergraph $\mathcal{H}$ has at least three hyperedges which intersect pairwise and they have at most one variable in common. 
\end{clm}
\begin{proof}
To show this we consider $e,e'$ and $e''$ from the above. If they have at most one variable in common we are done so we assume they have two variables in common say $w_1,w_2$ one of $w_1$ or $w_2$ (say $w_1$) must be squarefree in $X^{a}$ (i.e. 
$w_1^2$ does not divide $X^{a}$). Hence from Lemma~\ref{sulivan}
and by using the fact that $w_1$ is a squarefree factor of $X^{a}$ we can say that there is an edge $e_{0}$ dividing $\frac{X^{a}}{w_1}$
which is not one of $e,e',e''$ (since otherwise $X^a\in I^2$ which is a contradiction). Now note that we can say $e_{0}$ intersects $e,e',e''$ and also since $w_1\notin e_0$ they have at most one variable in common. So we just need to replace one of $e,e',e''$ by $e_{0}$ and it settles the proof. 
\end{proof}
Therefore, we have shown $e,e'$ and $e''$ have at most one variable in common. We consider the following cases: 
\begin{enumerate}
\item Suppose $e,e'$ and $e''$ have no variable in common. If two of them say $e$ and $e'$ have two variables in common we can write $e\cap e'=\{x,y\}$. Without loss of generality we assume $x\in X$ and $y\in Y$. Since the hypergraph $\mathcal{H}$ is 3-uniform and 3-partite we can conclude that $e=\{x,y,z_1\}$ and $e'=\{x,y,z_2\}$ such that $z_1,z_2\in Z$. Since we know $e''$ intersects $e$ and $e'$ and $e,e',e''$ have no vertices in common we can conclude that $z_1,z_2\in e''$ which is a contradiction because $\mathcal{H}$ is 3-partite. Therefore 
$$|e\cap e'|=|e\cap e''|=|e'\cap e''|=1$$ 
and since $e,e'$ and $e''$ have no variable in common we can write $e\cap e'\neq e\cap  e''\neq e'\cap e''$.  So it is clear that the subhypergraph $\left\langle e,e',e''\right\rangle$ of $\mathcal{H}$ is the bad hypergraph of length 3. 
\item Suppose $e,e'$ and $e''$ have exactly one variable in common say $w$. Now we consider the followings
\begin{enumerate}
\item Suppose $e\cap e'=e\cap e''=e'\cap e''=\{w\}.$ Again since $X^{a}\notin I^{2}$ we have $w^2$ doesn't divide $X^a$ and since $X^{a}\in I^{(2)}$ from Lemma~\ref{sulivan} there is $e_0\in I$ such that $e_0$ divides $X^{a}/w$. Since $w^2$ doesn't divide $X^a$ we can write $w\notin e_0$ and then $e_0\notin \{e,e',e''\}$. Also since $X^{a}\notin I^2$ we can conclude that $e_0$ intersects $e,e'$ and $e''$. Since $e,e',e''$ have just one variable in common (i.e $w$) and $w\notin e_0$ if we consider $e_0$ with two of $e,e',e''$ (say $e,e'$) we can conclude that $e,e',e_0$ are three hyperedges in $\mathcal{H}$ which have no variable in common. Therefore, from part (1) we can say the subhypergraph   $\left\langle e,e',e_0\right\rangle$ of $\mathcal{H}$ is the bad hypergraph of length 3.
\item Suppose $|e\cap e'|=|e\cap e''|=|e'\cap e''|=2$. Since $\mathcal{H}$ is 3-partite and $e,e'$ and $e''$ have exactly one vertex in common then this case never happens. So we assume at least one couple of $e,e',e''$ intersecting each other at just $w$ (say $e\cap e''=\{w\}$). We assume $e\cap e'=\{w,z_1\}$. Note that $w^2$ does not divide $X^{a}$ because otherwise $ee''$ divides $X^{a}$ and then $X^{a}\in I^{2}$ which is a contradiction. Hence, from Lemma~\ref{sulivan} we can say there is $e_0\in I$ dividing $\frac{X^{a}}{w}$. Obviously $e_0,e,e''$ have no variable in common therefore, from part (1) we can say the subhypergraph   $\left\langle e,e'',e_0\right\rangle$ of $\mathcal{H}$ is the bad hypergraph of length 3.   
\end{enumerate}

\end{enumerate} 
 \end{proof}
\begin{cor}\label{mycol}
Let $\mathcal{H}$ be a 3-uniform, 3-partite hypergraph and $I=I(\mathcal{H})$ be its edge ideal. Suppose $\mathcal{H}$ has no bad subhypergraph, then $I^{(2)}=I^2$. 
\end{cor}

\begin{thm}
Let $\mathcal{H}$ be a 3-uniform and 3-partite hypergraph and $I=I(\mathcal{H})$ be its edge ideal. Then $I^{(2)}=I^2$ if and only if one of the following holds
\begin{enumerate}
\item $\mathcal{H}$ has no bad subhypergraph of length 3;
\item If $\mathcal{H}$ has a bad subhypergraph of length 3, say $\mathcal{B}$, then there is a hyperedge $E$ of $\mathcal{B}$ such that $V(\mathcal{B})\backslash E$ is a hyperedge of $\mathcal{H}$
\end{enumerate}
\end{thm}
\begin{proof}
If part (1) holds then from Corollary~\ref{mycol} we have $I^{(2)}=I^2$. Then we assume part(2) holds, we will show $I^{(2)}=I^2$. Suppose $I^{(2)}\neq I^2$. Then there is a monomial $X^{a}\in I^{(2)}\backslash I^2$. From Observation~\ref{myobs} we can say there is a bad subhypergraph of length 3, say $\mathcal{B}$, such that 
$$X^{a}=\left(\prod_{x\in V(\mathcal{B})} x\right)u$$
where $u$ is a monomial. From part (2) we know there is a hyperedge $E$ in $\mathcal{B}$ such that 
$$\prod_{x\in V(\mathcal{B})\backslash V(E)} x\in I.$$ 
Then we have 
$$X^{a}=\left(\prod_{x\in V(\mathcal{B})} x\right)u=\left(\prod_{x\in V(E)} x\right)\left(\prod_{x\in V(\mathcal{B})\backslash V(E)} x\right)u\in I^{2}$$
and it is a contradiction.  

Now we suppose $I^{(2)}=I^2$. If there is no bad subhypergraph of length 3 we have nothing to show. So we assume $\mathcal{H}$ has a bad subhypergraph of length 3 say $\mathcal{B}$. We consider the monomial $X^a= \displaystyle\prod_{x\in V(\mathcal{B})} x$.
It is straightforward to see that $X^a\in I^{(2)}$. Since $I^{(2)}=I^2$, we can conclude that $X^a\in I^2$ and since $\text{deg}(X^a)=6$ we can write there is a hyperedge $E$ in $\mathcal{B}$ such that $V(\mathcal{B})\backslash E$ is a hyperedge of $\mathcal{H}$. It settles our claim.   

\end{proof}

\begin{col}
Let $\mathcal{H}$ be a 3-uniform, 3-partite hypergraph and $I=I(\mathcal{H})$ be its edge ideal. Then we have 
$$I^{(2)}=I^2+\left(\displaystyle\prod_{x\in V(\mathcal{B})}x:\text{$\mathcal{B}$ is a bad subhypergraph of length 3} \right) $$
\end{col}

%%%%%%%%%%%%%%%%%%%%%%%%%%%%%%%%%%%%%%%%%%%%%%%%%%%%%%%%%%%%%%%%%%%%%%%%%%%%%%%%%%%%%%%%
%%%%%%%%%%%%%%%%%%%%%%%%%%%%%%%%%%%%%%%%%%%%%%%%%%%%%%%%%%%%%%%%%%%%%%%%%%%%%%%%%%%%%%%%
\section{Path ideals}
%%%%%%%%%%%%%%%%%%%%%%%%%%%%%%%%%%%%%%%%%%%%%%%%%%%%%%%%%%%%%%%%%%%%%%%%%%%%%%%%%%%%%%%%
%%%%%%%%%%%%%%%%%%%%%%%%%%%%%%%%%%%%%%%%%%%%%%%%%%%%%%%%%%%%%%%%%%%%%%%%%%%%%%%%%%%%%%%%%

Path ideals first were introduced by Conca and De Negri in \cite{conca1999}. In this section we will study symbolic powers of a cubic path ideal of a graph. First we will recall the definition of the path ideal. 
\begin{defn}
Let $G=(V,E)$ be a simple graph. A sequence of distinct vertices $v_1,v_2,\dots,v_n$ is called a path if $\{v_i,v_{i+1}\}$ is an edge in $G$ for each $1\leq i<n$. The number of edges in a path is defined as the length of that path. We define the {\bf path ideal} of $G$, denoted by $I_t(G)$, to be the
ideal of $k[V]$ generated by the monomials of the form
$x_{i_1}x_{i_2}\dots x_{i_t}$ where $x_{i_1},x_{i_2},\dots,x_{i_t}$ is
a path in $G$. 
\end{defn}
\begin{defn}
Let $G=(V,E)$ be a simple graph and $I_t(G)$ be a path ideal of length $t-1$ when $2\leq t$. Then we define the {\bf path hypergraph} as a hypergraph whose vertices are $V$ and whose edges consist of all paths of the length $t-1$. This hypergraph is denoted by $\mathcal{H}_t(G)$.
\end{defn}
%We also need the following result from Jing Jane He, Adam Van Tuyl.
%\begin{thm}
%Let $t\geq 2$ be an integer and $G$ be a rooted tree. The path hypergraph  $\mathcal{H}_t(G)$ is a simplicial tree. 
%\end{thm}
From Theorem~\ref{Tai}, [\cite{Adam2010}, Theorem 2.7] and [\cite{Herzog2008}, Theorem 3.2] we have the following result. 
\begin{col}\label{Adam'thm}
Let $t\geq 2$ be an integer and $G$ be a rooted tree (i.e. A rooted tree is a tree with one vertex chosen as a root). The path hypergraph  $\mathcal{H}_t(G)$ is a simplicial tree (for details about simplicial trees see \cite{Faridi2011}). Moreover for the path ideal $J=I_t(G)$ we have $J^{(n)}=J^{n}$ for all $n\geq 1$.  
\end{col}
Now we are ready to show when the $t$-path hypergraph of a cycle is $t$-partite.

\begin{thm}\label{partitePath}
Let $n\geq 2$ and $2\leq t\leq n$ be integers and $C_n$ be a cycle over $\{x_1,\dots,x_n\}$. Suppose $\mathcal{H}_{t}(C_n)$ be the $t$-path hypergraph of $C_n$. Then we have $\mathcal{H}_{t}(C_n)$ is a $t$-partite hypergraph if and only if $n\equiv 0 \pmod{t}$.    
\end{thm}
\begin{proof} We assume $\mathcal{H}_{t}(C_n)$ is a $t$-partite hypergraph with the $t$-partition $(X_1,\dots,X_t)$. 
Since $\{x_1,x_2,\dots,x_{t}\}$ is a hyperedge of $\mathcal{H}_{t}(C_n)$, without loss of generality we can write $x_i\in X_i$ for $i=1,2,\dots, t$. We will show that for each positive integer $g\leq n$ if 
$g\not\equiv 0\pmod{t}$ we have $x_{g}\in X_i$ when $g\equiv i \pmod{t}$  and otherwise $x_g\in X_t$.    

If we assume $g\leq t$, it is obvious. So we assume that $g>t$. We proceed the proof by the induction on $g$. If $g=t+1$ since we have $\{x_2,x_3,\dots,x_{t+1}\}$ is a hyperedge of $\mathcal{H}_{t}(C_n)$ and because $x_i\in X_i$ for $i=2,3,\dots, t$ then we can conclude that $x_{t+1}\in X_1$ and it settles the base of the induction. 

We assume $g>t+1$ and $g\equiv a \pmod{t}$ where $1 \leq a\leq t$. 
We consider the hyperedge $E=\{x_{g-t+1},\dots,x_g\}$ of $\mathcal{H}_{t}(C_n)$. By the induction hypotheses we can write $x_{g-j}\in X_{k_j}$ for $j=1,2,\dots,t-1$ when $g-j\equiv k_j \pmod{t}$ and $k_j\in \{1,2,\dots,t\}=[t]$. Now note that for each $j$ we can say $k_j\neq a$, because otherwise there is $1\leq j\leq t$ such that 
$a+j\equiv a \pmod{t}$ and then $j\equiv 0 \pmod{t}$ and it is a contradiction. 
Then since $\mathcal{H}_{t}(C_n)$ is $t$-partite and $E$ is a hyperedge of $\mathcal{H}_t(C_n)$ we can write the $k_j$s are distinct and thus 
we can conclude that $x_g\in X_a$ and it settles our claim. 

On the other hand since $\{x_1,x_2,\dots,x_{t-1},x_n\}$ is a hyperedge of $\mathcal{H}_{t}(C_n)$ and $x_i\in X_i$ for each $i=1,2,\dots, t-1$ then because $\mathcal{H}_{t}(C_n)$ is $t$-partite we can conclude that $x_n\in X_t$ and then $n\equiv t\equiv 0 \pmod{t}$.   

\end{proof}
The following corollary is a direct conclusion of Theorem~\ref{Isabel} and  Theorem~\ref{partitePath}.  
\begin{col}
Let $n\geq 2$ and $2\leq t\leq n$ be integers and $C_n$ be a cycle over $\{x_1,\dots,x_n\}$ and let $\mathcal{H}_{t}(C_n)$ be  the $t$-path hypergraph of $C_n$ and $I=I_t(C_n)$ be the $t$-path ideal of $C_n$. If $I^{(r)}=I^{r}$ for each $r\geq 1$, then we have $n\equiv 0 \pmod{t}$. 
  
\end{col}
In the following result we use [\cite{Adm2016}, Theorem 4.6], to give a formula to compute the Waldschmidt constant of the path ideals of a cycle. 
\begin{lem}\label{easylemma}
Let $C_n$ be a cycle graph over the vertex set $V=\{x_1,x_2,\dots,x_n\}$ and $2\leq t\leq n$ be an integer and let $\mathcal{H}_{t}(C_n)$ be the path hypergraph of $C_n$, then $\mathcal{H}_{t}(C_n)$ is a vertex-transitive hypergraph (see Definition~\ref{transitivity}). 
\end{lem} 
\begin{proof}
Pick $x_r,x_m\in V$. We define the automorphism $\pi:V\longrightarrow V$ on $\mathcal{H}_{t}(C_n)$ is given by $\pi(x_{i})=x_{a_i}$ where $x_{a_i}\in V$ and  
$a_i \equiv i+m-r \pmod{n}$. We clearly have $\pi(x_r)=x_m$ and it settles our claim.  

\end{proof} 
Now we are ready to show the following proposition. 
\begin{prop}
Let $C_n$ be a cycle graph over the vertex set $V=\{x_1,x_2,\dots,x_n\}$ and $2\leq t\leq n$ be an integer. Suppose $n=tq+r$ for $r=0,1,\dots, t-1$ and put $I=I_{t}(C_n)$. Then we have 
\[ \widehat{\alpha}(I)=\begin{cases} 
      t & r=0 \\
      \displaystyle\frac{n}{q+1} & r\neq 0 
   \end{cases}
\] 
\end{prop}
\begin{proof}
By Lemma \ref{easylemma} we have $\mathcal{H}_{t}(C_n)$ is vertex-transitive (Definition~\ref{transitivity}), then by using [\cite{Ullman2011}, Proposition 3.1.1] we have 
$\chi^{*}(\mathcal{H}_{t}(C_n))=\displaystyle\frac{n}{\a}$ where $\a=\a(\mathcal{H}_{t}(C_n))$. If $r=0$, then it is straightforward to see that the independent set   
$W=V\backslash\{x_{ti}:i=1,2,\dots,q\}$  has the maximum size. Therefore, since $|W|=n-q$ we can write $\a=n-q$. Hence   
\begin{eqnarray}\label{number1}
\chi^{*}(\mathcal{H}_{t}(C_n))=\displaystyle\frac{n}{n-q}&\text{if $r=0$}.
\end{eqnarray}
If we assume $r\neq 0$ we can write $W=V\backslash\{x_{ti},x_n:i=1,2,\dots,q\}$ is an independent set with the maximum size. Therefore we have $\a=n-q-1$ and so we can write
\begin{eqnarray}\label{number2}    
\chi^{*}(\mathcal{H}_{t}(C_n))=\displaystyle\frac{n}{n-q-1}&\text{if $r\neq 0$}.
\end{eqnarray}
By using (\ref{number1}), (\ref{number2}) and by using [\cite{Adm2016}, Theorem 4.6]  we can conclude that our claim is true. 
\end{proof}

%%%%%%%%%%%%%%%%%%%%%%%%%%%%%%%%%%%%%%%%%%%%%%%%%%%%%%%%%%%%%%%%%%%%%%%%%%%%%%%
\section*{Acknowledgment}
%%%%%%%%%%%%%%%%%%%%%%%%%%%%%%%%%%%%%%%%%%%%%%%%%%%%%%%%%%%%%%%%%%%%%%%%%%%%%%%%

We are very thankful to Professor R. Villarreal and Professor D. Ullman for their helpful comments and suggestions.  We gratefully acknowledge the helpful computer algebra system Macaulay2~\cite{M2} , without which our work would have been difficult or impossible. 

\bibliographystyle{plain}
\bibliography{mine1}

\end{document}